\documentclass{article}
\usepackage{amsfonts}
\usepackage{amssymb}
\usepackage{amsmath}

\newtheorem{theorem}{Theorem}

\newtheorem{corollary}{Corollary}

\newtheorem{lemma}{Lemma}

\newtheorem{remark}{Remark}

\newenvironment{proof}[1][Proof]{\noindent\textbf{#1.} }{\ \rule{0.5em}{0.5em}}

\begin{document}

\title{\textbf{Pointwise analog of the Ste\v{c}kin approximation theorem}}
\author{\textbf{W\l odzimierz \L enski } \\
University of Zielona G\'{o}ra\\
Faculty of Mathematics, Computer Science and Econometrics\\
65-516 Zielona G\'{o}ra, ul. Szafrana 4a\\
P O L A N D\\
W.Lenski@wmie.uz.zgora.pl }
\date{}
\maketitle

\begin{abstract}
We show the pointwise version of the Ste\v{c}kin theorem on approximation by
de la Vall\'{e}e-Poussin means. The result on norm approximation is also
derived.

\ \ \ \ \ \ \ \ \ \ \ \ \ \ \ \ \ \ \ \ 


\ \ \ \ \ \ \ \ \ \ \ \ \ \ \ \ \ \ \ 

\end{abstract}

\newpage 

\section{Introduction}

Let $L^{p}\ (1\leq p<\infty )\ \left[ C\right] $ be the class of all $2\pi $%
--periodic real--valued functions integrable in the Lebesgue sense with $p$%
--th power $\left[ \text{continuous}\right] $ over $Q=$ $[-\pi ,\pi ]$ and
let $X^{p}=L^{p}$ when $1\leq p<\infty $ or $X^{p}=C$ when $p=\infty $.

Let us define the norms of $f\in X^{p}$ as\ \ \ \ 
\begin{equation*}
\Vert f\Vert =\Vert f\Vert _{_{X^{p}}}=\Vert f(\cdot )\Vert
_{_{X^{p}}}:=\left\{ 
\begin{array}{ll}
\left\{ \int_{_{_{Q}}}\mid f(x)\mid ^{p}dx\right\} ^{\frac{1}{p}} & \text{%
when }1\leq p<\infty \  \\ 
\sup_{x\in Q}\mid f(x)\mid  & \text{when \ }p=\infty 
\end{array}%
\right. 
\end{equation*}%
and%
\begin{eqnarray*}
\Vert f\Vert _{x,\delta } &=&\Vert f\Vert _{X^{p},x,\delta }=\Vert f(\cdot
)\Vert _{X^{p},x,\delta }:=\sup\limits_{0<h\leq \delta }\Vert f(\cdot )\Vert
_{X^{p},x,h}^{\circ } \\
&:&=\left\{ 
\begin{array}{ll}
\sup\limits_{0<h\leq \delta }\left\{ \frac{1}{2h}\int\limits_{x-h}^{x+h}\mid
f(t)\mid ^{p}dt\right\} ^{\frac{1}{p}} & \text{when }1\leq p<\infty \  \\ 
\sup\limits_{0<h\leq \delta }\left\{ \sup\limits_{0<\left\vert t\right\vert
\leq h}\mid f(x+t)\mid \right\}  & \text{when \ }p=\infty 
\end{array}%
\right. 
\end{eqnarray*}%
where%
\begin{eqnarray*}
\Vert f\Vert _{x,\delta }^{\circ } &=&\Vert f\Vert _{X^{p},x,\delta }^{\circ
}=\Vert f(\cdot )\Vert _{X^{p},x,\delta }^{\circ } \\
&:&=\left\{ 
\begin{array}{ll}
\left\{ \frac{1}{2h}\int\limits_{x-\delta }^{x+\delta }\mid f(t)\mid
^{p}dt\right\} ^{\frac{1}{p}} & \text{when }1\leq p<\infty  \\ 
\sup\limits_{0<\left\vert t\right\vert \leq h}\mid f(x+h)\mid  & \text{when
\ }p=\infty 
\end{array}%
\right. \text{ \ \ \ }\left( \delta >0\right) .
\end{eqnarray*}%
We note additionally that 
\begin{equation*}
\Vert f\Vert _{X,x,0}=\Vert f\Vert _{X,x,0}^{\circ }=\mid f(x)\mid .
\end{equation*}%
Consider the trigonometric Fourier series of $f$ 
\begin{equation*}
Sf(x)=\frac{a_{o}(f)}{2}+\sum_{k=0}^{\infty }(a_{k}(f)\cos kx+b_{k}(f)\sin
kx)
\end{equation*}%
with the partial sums \ $S_{k}f$\ \ .

Let 
\begin{equation*}
\sigma _{n,m}f\left( x\right) :=\frac{1}{m+1}\sum_{k=n-m}^{n}S_{k}f\left(
x\right) \text{ \ \ \ }\left( m\leq n=0,1,2,...\right)
\end{equation*}

As a measure of approximation by the above quantities we use the pointwise
characteristics%
\begin{eqnarray*}
w_{x}f(\delta ) &=&w_{x}f(\delta )_{X^{p}}:=\Vert \Delta _{x}f\left( \cdot
\right) )\Vert _{X^{p},x,\delta } \\
&=&\left\{ 
\begin{array}{ll}
\sup\limits_{0<h\leq \delta }\left\{ \frac{1}{2h}\int_{-h}^{h}\left\vert
\Delta _{x}f\left( t\right) \right\vert ^{p}dt\right\} ^{1/p} & \text{when }%
1\leq p<\infty  \\ 
\sup\limits_{0<h\leq \delta }\left\{ \sup\limits_{0<\left\vert t\right\vert
\leq h}\mid \Delta _{x}f\left( t\right) \mid \right\}  & \text{when \ }%
p=\infty 
\end{array}%
\right. 
\end{eqnarray*}%
cf. \cite{ABT}\ and 
\begin{eqnarray*}
w_{x}^{\circ }f(\delta ) &=&w_{x}^{\circ }f(\delta )_{X^{p}}:=\Vert \Delta
_{x}f\left( \cdot \right) )\Vert _{X^{p},x,\delta }^{\circ } \\
&=&\left\{ 
\begin{array}{ll}
\left\{ \frac{1}{2\delta }\int_{-\delta }^{\delta }\left\vert \Delta
_{x}f\left( t\right) \right\vert ^{p}dt\right\} ^{1/p} & \text{when }1\leq
p<\infty  \\ 
\sup\limits_{0<\left\vert t\right\vert \leq \delta }\mid \Delta _{x}f\left(
t\right) \mid  & \text{when \ }p=\infty 
\end{array}%
\right. 
\end{eqnarray*}%
and also%
\begin{equation*}
\Omega _{x}f\left( \frac{\pi }{n+1}\right) =\Omega _{x}f\left( \frac{\pi }{%
n+1}\right) _{X^{p}}:=\frac{1}{n+1}\sum_{k=0}^{n}w_{x}f(\frac{\pi }{k+1}%
)_{X^{p}}
\end{equation*}%
and%
\begin{equation*}
\Omega _{x}^{\circ }f\left( \frac{\pi }{n+1}\right) =\Omega _{x}^{\circ
}f\left( \frac{\pi }{n+1}\right) _{X^{p}}=\frac{1}{n+1}\sum_{k=0}^{n}w_{x}^{%
\circ }f(\frac{\pi }{k+1})_{X^{p}}\text{ ,}
\end{equation*}%
\begin{equation*}
\text{where \ \ }\Delta _{x}f\left( t\right) :=f\left( x+t\right) -f\left(
x\right) \text{ \ ,}
\end{equation*}%
constructed on the base of definition of $X^{p}-points\ \ (\left[ \text{%
Lebesgue points}(L^{p}-points)\right] $ or \ $\left[ \text{points of
continuity}\left( C-points\right) \right] ).$ We also use the modulus of
continuity of $f$ in the space $X^{p}$ defined by the formula 
\begin{equation*}
\omega f\left( \delta \right) =\omega f\left( \delta \right)
_{X^{p}}:=\sup_{0<\left\vert h\right\vert \leq \delta }\left\Vert \Delta
_{\cdot }f\left( h\right) \right\Vert _{X^{p}}\text{ \ }
\end{equation*}%
and its arithmetic mean%
\begin{equation*}
\Omega f\left( \frac{\pi }{n+1}\right) =\Omega f\left( \frac{\pi }{n+1}%
\right) _{X^{p}}=\frac{1}{n+1}\sum_{k=0}^{n}\omega f(\frac{\pi }{k+1}%
)_{X^{p}}.
\end{equation*}

We can observe that, for $f\in X^{\widetilde{p}}$ and $\widetilde{p}\geq p$\
, 
\begin{equation*}
\Vert w_{\cdot .}f(\delta )_{X^{p}}\Vert _{C}\leq \omega f\left( \delta
\right) _{C},
\end{equation*}%
whence%
\begin{equation*}
\Vert \Omega _{\cdot .}f(\delta )_{X^{p}}\Vert _{C}\leq \Omega f\left(
\delta \right) _{C}
\end{equation*}%
and%
\begin{equation*}
\Vert w_{\cdot .}^{\circ }f(\delta )_{X^{p}}\Vert _{X^{p}}\leq \omega
f\left( \delta \right) _{X^{p}},
\end{equation*}%
whence%
\begin{equation*}
\Vert \Omega _{\cdot .}^{\circ }f(\delta )_{X^{p}}\Vert _{X^{p}}\leq \Omega
f\left( \delta \right) _{X^{p}}.
\end{equation*}

Let introduce one more measure of poitwise approximation analogical to the
best approximation of function \ $f$ by trigonometric polynomials $T$\ \ of
the degree at most $n$ $\left( T\in H_{n}\right) $%
\begin{equation*}
E_{n}\left( f\right) _{X^{p}}:=\inf_{T\in H_{n}}\left\{ \left\Vert f\left(
\cdot \right) -T\left( \cdot \right) \right\Vert _{X^{p}}\right\} ,
\end{equation*}%
namely%
\begin{equation*}
E_{n}\left( f,x;\delta \right) =E_{n}\left( f,x;\delta \right)
_{X^{p}}:=\inf_{T\in H_{n}}\left\{ \left\Vert f\left( \cdot \right) -T\left(
\cdot \right) \right\Vert _{X^{p},x,\delta }\right\} 
\end{equation*}%
\begin{equation*}
=\left\{ 
\begin{array}{ll}
\inf\limits_{T\in H_{n}}\left\{ \sup\limits_{0<h\leq \delta }\left[ \frac{1}{%
2h}\int_{-h}^{h}\left\vert f\left( x+t\right) -T\left( x+t\right)
\right\vert ^{p}dt\right] ^{\frac{1}{p}}\right\}  & \text{when }1\leq
p<\infty  \\ 
\inf\limits_{T\in H_{n}}\left\{ \sup\limits_{0<\left\vert h\right\vert \leq
\delta }\mid f(x+h)-T\left( x+h\right) )\mid \right\}  & \text{when \ }%
p=\infty 
\end{array}%
\right. 
\end{equation*}%
and%
\begin{equation*}
E_{n}^{\circ }\left( f,x;\delta \right) =E_{n}^{\circ }\left( f,x;\delta
\right) _{X^{p}}:=\inf_{T\in H_{n}}\left\{ \left\Vert f\left( \cdot \right)
-T\left( \cdot \right) \right\Vert _{X^{p},x,\delta }^{\circ }\right\} .
\end{equation*}%
We will also use its arithmetic mean%
\begin{equation*}
F_{n,m}\left( f,x\right) =F_{n,m}\left( f,x\right) _{X}:=\frac{1}{m+1}%
\sum_{k=0}^{m}E_{n}\left( f,x;\frac{\pi }{k+1}\right) _{X}
\end{equation*}%
and%
\begin{equation*}
F_{n,m}^{\circ }\left( f,x\right) =F_{n,m}^{\circ }\left( f,x\right)
_{X^{p}}:=\frac{1}{m+1}\sum_{k=0}^{m}E_{n}^{\circ }\left( f,x;\frac{\pi }{k+1%
}\right) _{X^{p}}\text{ .}
\end{equation*}

Denote also 
\begin{equation*}
X^{p}\left( w_{x}\right) =\left\{ f\in X^{p}:w_{x}f\left( \delta \right)
\leq w_{x}\left( \delta \right) \right\} ,
\end{equation*}%
where $\ w_{x}$ is a function of modulus of continuity type on the interval $%
[0,+\infty ),$ i.e. a nondecreasing continuous function having the following
properties: \ $w_{x}\left( 0\right) =0,$ \ $w_{x}\left( \delta _{1}+\delta
_{2}\right) \leq w_{x}\left( \delta _{1}\right) +w_{x}\left( \delta
_{2}\right) $ \ for any \ \ $0\leq \delta _{1}\leq \delta _{2}\leq \delta
_{1}+\delta _{2}$.

Using these characteristics we will show the pointwise version of the Ste%
\v{c}kin \cite{S} generalization of the Fej\'{e}r-Lebesgue theorem. As a
corollaries we will obtain the mentioned original result of S. B. Ste\v{c}%
kin on norm approximation as well the result of N. Tanovi\'{c}-Miller \cite%
{TM}.

By $K$ we shall designate either an absolute constant or a constant
depending on some parameters, not necessarily the same of each occurrence.

\section{Statement of the results}

At the begin we formulate the partial solution of the considered problem.

\begin{theorem}
\cite{WL}If \ $f\in X^{p}$\ then, for any positive integer\ $m\leq n$ and
all real $x,$%
\begin{eqnarray*}
\left\vert \sigma _{n,m}f\left( x\right) -f\left( x\right) \right\vert &\leq
&\pi ^{2}E_{n-m}^{\circ }(f,x,\frac{\pi }{2n-m+1})_{X}+6F_{n-m,m}^{\circ
}\left( f,x\right) _{X} \\
&&+\int_{\frac{\pi }{m+1}}^{\frac{\pi }{2n-m+1}}\frac{E_{n-m}^{\circ
}(f,x,t)_{X}}{t}dt+E_{n-m}^{\circ }\left( f,x;0\right) _{X^{p}}
\end{eqnarray*}%
and%
\begin{eqnarray*}
\left\vert \sigma _{n,m}f\left( x\right) -f\left( x\right) \right\vert &\leq
&\left( 6\text{ }+\pi ^{2}\right) F_{n-m,m}\left( f,x\right) _{X^{p}}\left[
1+\ln \frac{n+1}{m+1}\right] \\
&&+E_{n-m}\left( f,x;0\right) _{X^{p}}.
\end{eqnarray*}
\end{theorem}

Now, we can present the main result on pointwise approximation.

\begin{theorem}
If \ $f\in X^{p}$\ \ then, for any positive integer \ $m\leq n$ and all real 
$x,$%
\begin{equation*}
\left\vert \sigma _{n,m}f\left( x\right) -f\left( x\right) \right\vert \leq
K\sum_{\nu =0}^{n}\frac{F_{n-m+\nu ,m}\left( f,x\right) _{X^{p}}+F_{n-m+\nu
,\nu }\left( f,x\right) _{X^{p}}}{m+\nu +1}+E_{2n}\left( f,x;0\right)
_{X^{p}}.
\end{equation*}
\end{theorem}

This immediately yields the following result of Ste\v{c}kin \cite{S}\ 

\begin{theorem}
If \ $f\in C$\ \ then, for any positive integer \ $n$ \ and \ $m\leq n$ \ 
\begin{equation*}
\left\Vert \sigma _{n,m}f\left( \cdot \right) -f\left( \cdot \right)
\right\Vert _{C}\leq K\sum_{\nu =0}^{n}\frac{E_{n-m+\nu }\left( f\right) _{C}%
}{m+\nu +1}\text{ .}
\end{equation*}
\end{theorem}

\begin{remark}
Theorem also holds if instead of \ $C$ \ we consider the spaces $X^{p}$ with
\ $1<p<\infty .$ In the proof we need the Hardy-Littlewood estimate of the
maximal function.
\end{remark}

At every $X^{p}-point\ $ $x$ of \ $f$%
\begin{equation*}
\Omega _{x}f\left( \gamma \right) _{X^{p}}=o_{x}\left( 1\right) \text{ \ \
as \ }\gamma \rightarrow 0+
\end{equation*}%
and thus from Theorem 1 we obtain the corollary which state the result of
the Tanovi\'{c}-Miller type\cite{TM}.

\begin{corollary}
If \ $f\in X^{p}$\ \ then, for any positive integer \ $m\leq n$ \ at every $%
X^{p}-point$ $x$ of \ $f$,%
\begin{equation*}
\left\vert \sigma _{n,m}f\left( x\right) -f\left( x\right) \right\vert
=o_{x}\left( 1\right) \left[ 1+\ln \frac{n+1}{m+1}\right] \text{ \ as \ }%
n\rightarrow \infty \text{ .}
\end{equation*}
\end{corollary}

\section{Auxiliary results}

In order to proof our theorems we require some lemmas

\begin{lemma}
If $T_{n}$\ is the trigonometric polynomial of the degree at most \ $n$ \ of
the best approximation of \ $f\in X^{p}$ \ with respect to the norm \ $%
\left\Vert \cdot \right\Vert _{X^{p}}$ then, it is also the trigonometric
polynomial of the degree at most \ $n$ \ of the best approximation of \ $%
f\in X^{p}$ \ with respect to the norm \ $\left\Vert \cdot \right\Vert
_{X^{p},x,\delta }$ for any $\delta \in \left[ 0,\pi \right] .$
\end{lemma}

\begin{proof}
From the inequalities%
\begin{eqnarray*}
\left\Vert E_{n}\left( f,\cdot ,\delta \right) _{X^{p}}\right\Vert _{X^{p}}
&\geq &\left\Vert E_{n}^{\circ }\left( f,\cdot ,\delta \right)
_{X^{p}}\right\Vert _{X^{p}} \\
&=&\left\Vert \left\Vert f-T_{n,\delta }\right\Vert _{X^{p},\cdot ,\delta
}^{\circ }\right\Vert _{X^{p}}=\left\Vert f-T_{n,\delta }\right\Vert _{X^{p}}
\\
&\geq &\left\Vert f-T_{n}\right\Vert _{X^{p}}=E_{n}\left( f\right) _{X^{p}}
\end{eqnarray*}%
and%
\begin{equation*}
\left\Vert E_{n}^{\circ }\left( f,\cdot ,\delta \right) _{X^{p}}\right\Vert
_{X^{p}}\leq \left\Vert \left\Vert f-T_{n}\right\Vert _{X^{p},\cdot ,\delta
}^{\circ }\right\Vert _{X^{p}}=\left\Vert f-T_{n}\right\Vert
_{X^{p}}=E_{n}\left( f\right) _{X^{p}}\text{ ,}
\end{equation*}%
where \ $T_{n,\delta }$\ and \ $T_{n}$\ are the trigonometric polynomials of
the degree at most \ $n$ \ of the best approximation of \ $f\in X^{p}$ \
with respect to the norms \ $\left\Vert \cdot \right\Vert _{X^{p},x,\delta
}^{\circ }$ and $\left\Vert \cdot \right\Vert _{X^{p}}$ respectively$\ ,$ we
obtain relation%
\begin{equation*}
\left\Vert f-T_{n,\delta }\right\Vert _{X^{p}}=\left\Vert f-T_{n}\right\Vert
_{X^{p}}=E_{n}\left( f\right) _{X^{p}}\text{ ,}
\end{equation*}%
whence \ \ $T_{n,\delta }=T_{n}$ \ for any $\delta \in \left[ 0,\pi \right] $
by uniqueness of the trigonometric polynomial of the degree at most \ $n$ \
of the best approximation of \ $f\in X^{p}$ \ with respect to the norm \ $%
\left\Vert \cdot \right\Vert _{X^{p}}$ $\ $(see e.g.\cite{BN} p. 96)$.$We
can also observe that for such $T_{n}$ and any $h\in \left[ 0,\delta \right] 
$ \ 
\begin{equation*}
\left\Vert f-T_{n}\right\Vert _{X^{p},x,h}^{\circ }=E_{n}^{\circ }\left(
f,x,h\right) _{X^{p}}\leq E_{n}\left( f,x,\delta \right) _{X^{p}}\leq
\left\Vert f-T_{n}\right\Vert _{X^{p},x,\delta }\text{ \ .}
\end{equation*}%
Hence%
\begin{equation*}
E_{n}\left( f,x,\delta \right) _{Xp}=\left\Vert f-T_{n}\right\Vert
_{X^{p},x,\delta }\text{ \ }
\end{equation*}%
and our proof is complete.
\end{proof}

\begin{lemma}
If $n\in 
\mathbb{N}
_{0}$ and $\delta >0$ then $\ E_{n}\left( f,x;\delta \right) _{X^{p}}$ is
nonincreasing function of \ $n$ \ and nondecreasing function of $\delta .$%
These imply that for $m,n\in 
\mathbb{N}
$\ the function $\ F_{n,m}\left( f,x\right) _{X^{p}}$ is nonincreasing
function of \ $n$ and $m$ simultaneously.
\end{lemma}

\begin{proof}
The first part of our statement follows from the property of the norm $%
\left\Vert \cdot \right\Vert _{x,\delta }$ and supremum. The second part is
a consequence of the calculation%
\begin{eqnarray*}
\frac{F_{n,m+1}\left( f,x\right) _{X^{p}}}{F_{n,m}\left( f,x\right) _{X^{p}}}
&=&\frac{m+1}{m+2}\left( 1+\frac{E_{n}\left( f,x;\frac{\pi }{m+2}\right)
_{X^{p}}}{\sum_{k=0}^{m}E_{n}\left( f,x;\frac{\pi }{k+1}\right) _{X^{p}}}%
\right) \\
&\leq &\frac{m+1}{m+2}\left( 1+\frac{E_{n}\left( f,x;\frac{\pi }{m+1}\right)
_{X^{p}}}{\sum_{k=0}^{m}E_{n}\left( f,x;\frac{\pi }{m+1}\right) _{X^{p}}}%
\right) \\
&=&\frac{m+1}{m+2}\left( 1+\frac{1}{m+1}\right) =1.
\end{eqnarray*}
\end{proof}

\begin{lemma}
Let \ $m,n,q\in 
\mathbb{N}
_{0}$ such that \ $m\leq n$ and $q\geq m+1.$ If \ $f\in X^{p}$ then%
\begin{equation*}
\left\vert \sigma _{n+q,m}f\left( x\right) -\sigma _{n,m}f\left( x\right)
\right\vert \leq KF_{n-m,m}\left( f,x\right) _{X^{p}}\sum_{\nu =0}^{q-1}%
\frac{1}{m+\nu +1}.
\end{equation*}
\end{lemma}

\begin{proof}
It is clear that%
\begin{eqnarray*}
\sigma _{n,m}f\left( x\right) &=&\frac{1}{m+1}\sum_{k=n-m}^{n}\frac{1}{\pi }%
\int_{-\pi }^{\pi }f\left( x+t\right) D_{k}\left( t\right) dt \\
&=&\frac{1}{\pi }\int_{-\pi }^{\pi }f\left( x+t\right) V_{n,m}\left(
t\right) dt
\end{eqnarray*}%
where%
\begin{equation*}
V_{n,m}\left( t\right) =\frac{1}{m+1}\sum_{k=n-m}^{n}D_{k}\left( t\right) 
\text{ \ and \ \ }D_{k}\left( t\right) =\frac{\sin \frac{\left( 2k+1\right) t%
}{2}}{2\sin \frac{t}{2}}.
\end{equation*}%
Hence, by orthogonality of the trigonometric system,%
\begin{eqnarray*}
&&\sigma _{n+q,m}f\left( x\right) -\sigma _{n,m}f\left( x\right) \\
&=&\frac{1}{\pi }\int_{-\pi }^{\pi }\left[ f\left( x+t\right) -T_{n-m}\left(
x+t\right) \right] \left( V_{n+q,m}\left( t\right) -V_{n,m}\left( t\right)
\right) dt \\
&=&\frac{1}{\pi \left( m+1\right) }\sum_{k=n-m}^{n}\int_{-\pi }^{\pi }\left[
f\left( x+t\right) -T_{n-m}\left( x+t\right) \right] \left( D_{k+q}\left(
t\right) -D_{k}\left( t\right) \right) dt \\
&=&\frac{1}{\pi \left( m+1\right) }\sum_{k=n-m}^{n}\int_{-\pi }^{\pi }\left[
f\left( x+t\right) -T_{n-m}\left( x+t\right) \right] \\
&&\frac{\sin \frac{\left( 2k+2q+1\right) t}{2}-\sin \frac{\left( 2k+1\right)
t}{2}}{2\sin \frac{t}{2}}dt \\
&=&\frac{1}{\pi \left( m+1\right) }\sum_{k=n-m}^{n}\int_{-\pi }^{\pi }\left[
f\left( x+t\right) -T_{n-m}\left( x+t\right) \right] \frac{\sin \frac{qt}{2}%
\cos \frac{\left( 2k++q1\right) t}{2}}{\sin \frac{t}{2}}dt
\end{eqnarray*}%
with trigonometric polynomial $T_{n-m}$\ \ of the degree at most \ $n-m$ \
of the best approximation of \ $f$ $.$

Using the notations%
\begin{eqnarray*}
I_{1} &=&\left[ -\frac{\pi }{q},\frac{\pi }{q}\right] ,\text{ \ \ \ \ }I_{2}=%
\left[ -\frac{\pi }{m+1},-\frac{\pi }{q}\right] \cup \left[ \frac{\pi }{q},%
\frac{\pi }{m+1}\right] \\
I_{3} &=&\left[ -\pi ,-\frac{\pi }{m+1}\right] \cup \left[ \frac{\pi }{m+1}%
,\pi \right]
\end{eqnarray*}%
we get%
\begin{eqnarray*}
\sum &=&\frac{1}{\pi \left( m+1\right) }\sum_{k=n-m}^{n}\left(
\int\limits_{I_{1}}+\int\limits_{\text{\ }I_{2}}+\int\limits_{I_{3}}\right) %
\left[ f\left( x+t\right) -T_{n-m}\left( x+t\right) \right] \\
&&\frac{\sin \frac{qt}{2}\cos \frac{\left( 2k+q+1\right) t}{2}}{\sin \frac{t%
}{2}}dt \\
&=&\sum\nolimits_{1}+\sum\nolimits_{2}+\sum\nolimits_{3}.
\end{eqnarray*}%
and%
\begin{eqnarray*}
\sum\nolimits_{1} &\leq &\frac{1}{\pi \left( m+1\right) }\sum_{k=n-m}^{n}%
\int\limits_{I_{1}}\left\vert f\left( x+t\right) -T_{n-m}\left( x+t\right)
\right\vert qdt \\
&=&\frac{q}{\pi }\int\limits_{I_{1}}\left\vert f\left( x+t\right)
-T_{n-m}\left( x+t\right) \right\vert dt \\
&\leq &2E_{n-m}\left( f,x;\frac{\pi }{q}\right) _{X^{p}}
\end{eqnarray*}%
We next evaluate the sums \ $\sum_{2}$ and $\sum_{3}$ using the partial
integrating and \ Lemma 1. Thus%
\begin{eqnarray*}
\sum\nolimits_{2} &\leq &\int\limits_{\text{\ }I_{2}}\frac{\left\vert
f\left( x+t\right) -T_{n-m}\left( x+t\right) \right\vert }{t}dt \\
&=&2\left[ \frac{1}{2t}\int_{-t}^{t}\left\vert f\left( x+u\right)
-T_{n-m}\left( x+u\right) \right\vert du\right] _{t=\frac{\pi }{q}}^{t=\frac{%
\pi }{m+1}} \\
&&+2\int_{\frac{\pi }{q}}^{\frac{\pi }{m+1}}\frac{1}{t}\left[ \frac{1}{2t}%
\int_{-t}^{t}\left\vert f\left( x+u\right) -T_{n-m}\left( x+u\right)
\right\vert du\right] dt
\end{eqnarray*}%
\begin{eqnarray*}
&\leq &2E_{n-m}\left( f,x;\frac{\pi }{m+1}\right) _{X^{p}}+2\int_{-t}^{t}%
\frac{1}{t}E_{n-m}\left( f,x;t\right) _{X^{p}}dt \\
&\leq &4E_{n-m}\left( f,x;\frac{\pi }{m+1}\right) _{X^{p}}\left[ 1+\ln \frac{%
q}{m+1}\right] \\
&\leq &4E_{n-m}\left( f,x;\frac{\pi }{m+1}\right) _{X^{p}}\left[ 1+\sum_{\nu
=0}^{q-1}\frac{1}{m+\nu +1}\right]
\end{eqnarray*}%
and%
\begin{eqnarray*}
\sum\nolimits_{3} &\leq &\frac{1}{m+1}\int\limits_{I_{3}}\frac{\left\vert
f\left( x+t\right) -T_{n-m}\left( x+t\right) \right\vert }{t} \\
&&\left\vert \sum_{k=n-m}^{n}\cos \left( kt+\frac{q+1}{2}t\right)
\right\vert dt
\end{eqnarray*}%
\begin{eqnarray*}
&\leq &\frac{1}{m+1}\int\limits_{I_{3}}\frac{\left\vert f\left( x+t\right)
-T_{n-m}\left( x+t\right) \right\vert }{t} \\
&&\left\vert \frac{2\sin \frac{\left( n+1\right) t}{2}\cos \frac{\left(
2n-m+q+1\right) t}{2}}{2\sin \frac{t}{2}}\right\vert dt
\end{eqnarray*}%
\begin{eqnarray*}
&\leq &\frac{\pi }{m+1}\int\limits_{I_{3}}\frac{\left\vert f\left(
x+t\right) -T_{n-m}\left( x+t\right) \right\vert }{t^{2}}dt \\
&=&\frac{\pi }{m+1}\left\{ 2\left[ \frac{1}{2t}\int_{-t}^{t}\left\vert
f\left( x+u\right) -T_{n-m}\left( x+u\right) \right\vert du\right] _{t=\frac{%
\pi }{m+1}}^{t=\pi }\right. \\
&&+4\int_{\frac{\pi }{m+1}}^{\pi }\frac{1}{t^{2}}\left[ \frac{1}{2t}%
\int_{-t}^{t}\left\vert f\left( x+u\right) -T_{n-m}\left( x+u\right)
\right\vert du\right] dt
\end{eqnarray*}%
\begin{eqnarray*}
&\leq &\frac{\pi }{m+1}\left\{ 2E_{n-m}\left( f,x;\pi \right)
_{X^{p}}+4\int_{\frac{\pi }{m+1}}^{\pi }\frac{1}{t^{2}}E_{n-m}\left(
f,x;t\right) _{X^{p}}dt\right\} \\
&=&\frac{2\pi }{m+1}\left\{ E_{n-m}\left( f,x;\pi \right)
_{X^{p}}+2\int_{1}^{m+1}\frac{E_{n-m}\left( f,x;\frac{\pi }{u}\right)
_{X^{p}}}{\pi ^{2}/u^{2}}\frac{\pi du}{u^{2}}\right\} \\
&=&\frac{2\pi }{m+1}\left\{ E_{n-m}\left( f,x;\pi \right) _{X^{p}}+\frac{2}{%
\pi }\sum_{k=0}^{m-1}\int_{k+1}^{k+2}E_{n-m}\left( f,x;\frac{\pi }{u}\right)
_{X^{p}}du\right\}
\end{eqnarray*}%
\begin{eqnarray*}
&=&\frac{2\pi }{m+1}\left\{ E_{n-m}\left( f,x;\pi \right) _{X^{p}}+\frac{2}{%
\pi }\sum_{k=0}^{m-1}E_{n-m}\left( f,x;\frac{\pi }{k+1}\right)
_{X^{p}}\right\} \\
&\leq &\frac{2\pi +\frac{2}{\pi }}{m+1}\sum_{k=0}^{m-1}E_{n-m}\left( f,x;%
\frac{\pi }{k+1}\right) _{X^{p}}
\end{eqnarray*}%
which proves Lemma 2.
\end{proof}

Before formulating the next lemmas \ we define a new difference. Let $m,n\in 
\mathbb{N}
_{0}$ \ and \ $m\leq n.$ \ Denote 
\begin{equation*}
\tau _{n,m}f\left( x\right) :=\left( m+1\right) \left\{ \sigma
_{n+m+1,m}f\left( x\right) -\sigma _{n,m}f\left( x\right) \right\} .
\end{equation*}

\begin{lemma}
Let \ $m,n,\mu \in 
\mathbb{N}
_{0}$ such that \ $2\mu \leq m\leq n$ $.$ If \ $f\in X^{p}$ then%
\begin{equation*}
\left\vert \tau _{n,m}f\left( x\right) -\tau _{n-\mu ,m-\mu }f\left(
x\right) \right\vert \leq K\mu F_{n-\mu +1,\mu -1}\left( f,x\right)
_{X^{p}}\ln \frac{m}{\mu }.
\end{equation*}
\end{lemma}

\begin{proof}
The proof follows by the method of Leindler\cite{L}. Namely%
\begin{equation*}
\tau _{n,m}f\left( x\right) -\tau _{n-\mu ,m-\mu }f\left( x\right) =\left(
\sum_{k=n+m-2\mu +2}^{n+m+1}-2\sum_{k=n-\mu +1}^{n}\right) \left[
S_{k}f\left( x\right) -f\left( x\right) \right]
\end{equation*}%
and%
\begin{equation*}
\left\vert \tau _{n,m}f\left( x\right) -\tau _{n-\mu ,m-\mu }f\left(
x\right) \right\vert \leq
\end{equation*}%
\begin{eqnarray*}
&\leq &\left\vert \left( \sum_{k=n+m-2\mu +2}^{n+m-\mu +1}-\sum_{k=n-\mu
+1}^{n}\right) \left[ S_{k}f\left( x\right) -f\left( x\right) \right]
\right\vert \\
&&+\left\vert \left( \sum_{k=n+m-\mu +2}^{n+m+1}-\sum_{k=n-\mu
+1}^{n}\right) \left[ S_{k}f\left( x\right) -f\left( x\right) \right]
\right\vert
\end{eqnarray*}%
\begin{equation*}
=\mu \left\vert \sigma _{n+m-\mu +1,\mu -1}f\left( x\right) -\sigma _{n,\mu
-1}f\left( x\right) \right\vert
\end{equation*}%
\begin{equation*}
+\mu \left\vert \sigma _{n+m+1,\mu -1}f\left( x\right) -\sigma _{n,\mu
-1}f\left( x\right) \right\vert .
\end{equation*}%
By Lemma 2 , for \ $2\mu \leq m,$%
\begin{equation*}
\left\vert \tau _{n,m}f\left( x\right) -\tau _{n-\mu ,m-\mu }f\left(
x\right) \right\vert \leq
\end{equation*}%
\begin{eqnarray*}
&\leq &K\mu F_{n-\mu +1,\mu -1}\left( f,x\right) _{X^{p}}\left[ 1+\ln \frac{%
\left( n-\mu +1\right) +\mu -1}{\mu }\right] \\
&&+K\mu F_{n-\mu +1,\mu -1}\left( f,x\right) _{X^{p}}\left[ 1+\ln \frac{%
m+\mu -1}{\mu }\right]
\end{eqnarray*}%
\begin{equation*}
\leq K\mu F_{n-\mu +1,\mu -1}\left( f,x\right) _{X^{p}}\left[ 1+\ln \frac{m}{%
\mu }\right]
\end{equation*}%
and our proof is complete.
\end{proof}

\begin{lemma}
Let \ $m,n\in 
\mathbb{N}
_{0}$ and\ $m\leq n$ $.$ If \ $f\in X^{p}$ then%
\begin{equation*}
\left\vert \tau _{n,m}f\left( x\right) \right\vert \leq
K\sum_{k=n-m}^{n}F_{k,k-n+m}\left( f,x\right) _{X^{p}}.
\end{equation*}
\end{lemma}

\begin{proof}
Our proof runs parallel with the proof of Theorem 1 in \cite{S}.

If \ $m=0$ then%
\begin{equation*}
\left\vert \tau _{n,0}f\left( x\right) \right\vert =\left\vert \sigma
_{n+1,0}f\left( x\right) -\sigma _{n,0}f\left( x\right) \right\vert \leq
KF_{n,0}\left( f,x\right) _{X^{p}}.
\end{equation*}%
and if \ $m=1$ then 
\begin{eqnarray*}
\left\vert \tau _{n,1}f\left( x\right) \right\vert &\leq &2\left\vert \sigma
_{n+1,1}f\left( x\right) -\sigma _{n,1}f\left( x\right) \right\vert \leq
KF_{n-1,1}\left( f,x\right) _{X^{p}} \\
&\leq &K\left[ F_{n-1,1}\left( f,x\right) _{X^{p}}+F_{n-1,1}\left(
f,x\right) _{X^{p}}\right] \\
&\leq &K\left[ F_{n-1,1}\left( f,x\right) _{X^{p}}+F_{n,1}\left( f,x\right)
_{X^{p}}\right]
\end{eqnarray*}%
by Lemma 2 and Lemma 3

Next we construct the same decreasing sequence $\left( m_{s}\right) $.of
integers that was given by S. B. Ste\v{c}kin. Let 
\begin{equation*}
m_{0}=m,\text{ \ }m_{s}=m_{s-1}-\left[ \frac{m_{s-1}}{2}\right] \text{ \ \ \ 
}\left( s=1,2,...\right)
\end{equation*}%
where $\left[ y\right] $ \ denotes the integral part of \ $y$. It is clear
that there exists a smallest index \ $t\geq 1$ such that \ $m_{t}=1$ and 
\begin{equation*}
m=m_{0}>m_{1}>...>m_{t}=1.
\end{equation*}

By the definition of the numbers \ $m_{s}$ we have 
\begin{equation*}
m_{s}\geq m_{s-1}/2
\end{equation*}%
\begin{equation*}
m_{s-1}-m_{s}=\left[ \frac{m_{s-1}}{2}\right] \geq \left[ \frac{m_{s-1}}{3}%
\right] \text{ \ \ \ }\left( s=1,2,...,t\right)
\end{equation*}%
whence%
\begin{equation*}
m_{t-1}=2,\text{ \ }m_{t-1}-m_{t}=1
\end{equation*}%
and%
\begin{equation*}
m_{s-1}-m_{s}\leq m_{s}\leq 3\left( m_{s}-m_{s+1}\right) \text{ \ \ \ }%
\left( s=1,2,...,t-1\right)
\end{equation*}%
follow.

Under these notations we get the following equality%
\begin{equation*}
\tau _{n,m}f\left( x\right) =\sum_{s=1}^{t}\left( \tau
_{n-m+m_{s-1},m_{s-1}}f\left( x\right) -\tau _{n-m+m_{s},m_{s}}f\left(
x\right) \right) +\tau _{n-m+m_{t},m_{t}}f\left( x\right)
\end{equation*}%
whence, by \ $m_{t}=1,$%
\begin{equation*}
\left\vert \tau _{n,m}f\left( x\right) \right\vert \leq
\sum_{s=1}^{t}\left\vert \tau _{n-m+m_{s-1},m_{s-1}}f\left( x\right) -\tau
_{n-m+m_{s},m_{s}}f\left( x\right) \right\vert +\left\vert \tau
_{n-m+m_{t},m_{t}}f\left( x\right) \right\vert
\end{equation*}%
follows.

It is easy to see that the terms in the sum $\sum_{s=1}^{t}$ , by Lemma 4,
with $\mu =m_{s-1}-m_{s}$ and \ $m=m_{s-1}$ do not exceed 
\begin{eqnarray*}
&&K\left( m_{s-1}-m_{s}\right) F_{n-m+m_{s}+1,m_{s-1}-m_{s}}\left(
f,x\right) _{X^{p}}\ln \frac{m_{s-1}}{m_{s-1}-m_{s}}, \\
&&\text{where }\left( s=1,2,...,t-1\right) .
\end{eqnarray*}%
and by Lemma 3 we get 
\begin{equation*}
\left\vert \tau _{n-m+1,1}f\left( x\right) \right\vert \leq 2\left\vert
\sigma _{n-m+2,1}f\left( x\right) -\sigma _{n-m+1,1}f\left( x\right)
\right\vert \leq KF_{n-m,1}\left( f,x\right) _{X^{p}}
\end{equation*}%
Thus%
\begin{eqnarray*}
\left\vert \tau _{n,m}f\left( x\right) \right\vert &\leq
&K\sum_{s=1}^{t-1}3\left( m_{s}-m_{s+1}\right) F_{n-m+m_{s}+1,m_{s}}\left(
f,x\right) _{X^{p}}\ln 3 \\
&&+KF_{n-m+2,m-2}\left( f,x\right) _{X^{p}}+KF_{n-m,1}\left( f,x\right)
_{X^{p}}
\end{eqnarray*}%
whence, by the monotonicity of \ $F_{\nu ,\mu }\left( f,x\right) _{X^{p}},$%
\begin{eqnarray*}
&&\left\vert \tau _{n,m}f\left( x\right) \right\vert \\
&\leq &K\left( \sum_{s=1}^{t-1}\sum_{\nu =m_{s+1}+1}^{m_{s}}F_{n-m+\nu
+1,\nu }\left( f,x\right) _{X^{p}}+\sum_{\nu =0}^{2}F_{n-m+\nu ,m-\nu
-1}\left( f,x\right) _{X^{p}}\right) \\
&&+KF_{n-m,1}\left( f,x\right) _{X^{p}}
\end{eqnarray*}%
\begin{eqnarray*}
&\leq &K\sum_{\nu =0}^{m_{1}+1}F_{n-m+\nu ,\nu }\left( f,x\right)
_{X^{p}}+KF_{n-m,1}\left( f,x\right) _{X^{p}} \\
&\leq &K\sum_{\nu =0}^{m}F_{n-m+\nu ,\nu }\left( f,x\right)
_{X^{p}}+KF_{n-m,1}\left( f,x\right) _{X^{p}} \\
&\leq &K\sum_{k=n-m}^{n}F_{k,k-n+m}\left( f,x\right)
_{X^{p}}+KF_{n-m,1}\left( f,x\right) _{X^{p}}\text{ \ .}
\end{eqnarray*}
\end{proof}

\section{Proofs of the results}

\subsection{Proof of Theorem 2}

The proof follows the lines of the proofs of Theorem 4 in \cite{S} and
Theorem in \cite{L}. Therefore let $n>0$ and $m\leq n$ be fixed. Let us
define an increasing sequence $\left( n_{s}:s=0,1,...,t\right) $ of indices
introduced by S. B. Ste\v{c}ki in the following way. Set $n_{0}=n.$ Assuming
that the numbers \ $n_{0},...,n_{s}$ are already defined and \ $n_{s}<2n,$
we define $n_{s+1}$ as follows: Let \ $\nu _{s}$ denote the smallest natural
number such that 
\begin{equation*}
F_{n_{s}-m+\nu _{s},\nu }\left( f,x\right) _{X^{p}}\leq \frac{1}{2}%
F_{n_{s}-m,\nu }\left( f,x\right) _{X^{p}}\text{ }\left( \nu
=0,1,...,n\right) \text{.}
\end{equation*}%
According to the magnitude of \ $\nu _{s}$ \ we define%
\begin{equation*}
n_{s+1}=\left\{ 
\begin{array}{cc}
n_{s}-m+1 & \text{ \ \ for \ }\nu _{s}\leq m\text{\ ,} \\ 
n_{s}+\nu _{s}\text{ } & \text{for \ \ }m+1\leq \nu _{s}<2n+m-n_{s}\text{ ,}
\\ 
2n+m & \text{ \ \ for \ }\nu _{s}\geq 2n+m-n_{s}\text{ }%
\end{array}%
\right.
\end{equation*}%
If \ $n_{s+1}<2n$ \ we continue the procedure, and if once \ $n_{s+1}\geq 2n$
\ then we stop the construction and define \ $t:=s+1.$

By the above definition of \ $\left( n_{s}\right) $ we have the following
obvious properties:%
\begin{equation*}
t\geq 1,\text{ \ }n=n_{0}<n_{1}<...<n_{t},\text{ \ \ }2n\leq n_{t}\leq 2n+m,
\end{equation*}%
and 
\begin{equation*}
n_{s+1}-n_{s}\geq m+1\text{ \ \ }\left( s=0,1,...,t-1\right) ,
\end{equation*}%
and relations%
\begin{equation*}
F_{n_{s+1}-m,\nu }\left( f,x\right) _{X^{p}}\leq \frac{1}{2}F_{n_{s}-m,\nu
}\left( f,x\right) _{X^{p}}\text{ \ \ for \ }s=0,1,...,t-2\text{,}
\end{equation*}%
and 
\begin{equation*}
\frac{1}{2}F_{n_{s}-m,\nu }\left( f,x\right) _{X^{p}}\leq F_{n_{s+1}-m-1,\nu
}\left( f,x\right) _{X^{p}}\text{ \ \ for \ }s=0,1,...,t-1
\end{equation*}%
whenever \ $n_{s+1}-n_{s}>m+1.$

Let us start with%
\begin{eqnarray*}
\left\vert \sigma _{n,m}f\left( x\right) -f\left( x\right) \right\vert
&=&\sum_{s=0}^{t-1}\left[ \left\vert \sigma _{n_{s},m}f\left( x\right)
-f\left( x\right) \right\vert -\left\vert \sigma _{n_{s+1},m}f\left(
x\right) -f\left( x\right) \right\vert \right] \\
&&+\left\vert \sigma _{n_{t},m}f\left( x\right) -f\left( x\right) \right\vert
\\
&\leq &\sum_{s=0}^{t-1}\left\vert \sigma _{n_{s+1},m}f\left( x\right)
-\sigma _{n_{s},m}f\left( x\right) \right\vert +\left\vert \sigma
_{n_{t},m}f\left( x\right) -f\left( x\right) \right\vert \\
&=&\sum_{s=0}^{t-1}\left\vert \frac{1}{m+1}\tau _{n_{s},m}f\left( x\right)
\right\vert +\left\vert \sigma _{n_{t},m}f\left( x\right) -f\left( x\right)
\right\vert .
\end{eqnarray*}%
Using Theorem 1 and that $2n\leq n_{t}\leq 2n+m$\ \ we get%
\begin{eqnarray*}
\left\vert \sigma _{n_{t},m}f\left( x\right) -f\left( x\right) \right\vert
&\leq &KF_{n_{t}-m,m}\left( f,x\right) _{X^{p}}\left[ 1+\ln \frac{n_{t}+1}{%
m+1}\right] +\left\vert f\left( x\right) -T_{n_{t}-m}\left( x\right)
\right\vert \\
&\leq &K\sum_{\nu =0}^{n}\frac{F_{n-m+\nu ,m}\left( f,x\right) _{X^{p}}}{%
m+\nu +1}+\left\vert f\left( x\right) -T_{n_{t}-m}\left( x\right) \right\vert
\\
&\leq &K\sum_{\nu =0}^{n}\frac{F_{n-m+\nu ,m}\left( f,x\right)
_{X^{p}}+F_{n-m+\nu ,\nu }\left( f,x\right) _{X^{p}}}{m+\nu +1}+\left\vert
f\left( x\right) -T_{n_{t}-m}\left( x\right) \right\vert \\
&\leq &K\sum_{\nu =0}^{n}\frac{F_{n-m+\nu ,m}\left( f,x\right)
_{X^{p}}+F_{n-m+\nu ,\nu }\left( f,x\right) _{X^{p}}}{m+\nu +1}%
+E_{n_{t}-m}\left( f,x;0\right) _{X^{p}} \\
&\leq &K\sum_{\nu =0}^{n}\frac{F_{n-m+\nu ,m}\left( f,x\right)
_{X^{p}}+F_{n-m+\nu ,\nu }\left( f,x\right) _{X^{p}}}{m+\nu +1}+E_{2n}\left(
f,x;0\right) _{X^{p}}.
\end{eqnarray*}

The estimate of the sum we derive from the following one%
\begin{equation*}
\left\vert \frac{1}{m+1}\tau _{n_{s},m}f\left( x\right) \right\vert \leq
K\sum_{\nu =0}^{n_{s+1}-n_{s}-1}\frac{F_{n_{s}-m+\nu ,m}\left( f,x\right)
_{X^{p}}+F_{n_{s}-m+\nu ,m}\left( f,x\right) _{X^{p}}}{m+\nu +1}.
\end{equation*}%
The proof of this inequality we split in two parts. If $\ n_{s+1}-n_{s}=m+1,$
then by Lemma 5, 
\begin{eqnarray*}
\left\vert \frac{1}{m+1}\tau _{n_{s},m}f\left( x\right) \right\vert &\leq &K%
\frac{1}{m+1}\sum_{k=n_{s}-m}^{n_{s}}F_{k,k-n_{s}+m}\left( f,x\right)
_{X^{p}} \\
&\leq &K\sum_{\nu =0}^{n_{s+1}-n_{s}-1}\frac{F_{n_{s}-m+\nu ,\nu }\left(
f,x\right) _{X^{p}}}{m+\nu +1}.
\end{eqnarray*}%
If $n_{s+1}-n_{s}>m+1,$ then, by Lemma 3, 
\begin{equation*}
\left\vert \frac{1}{m+1}\tau _{n_{s},m}f\left( x\right) \right\vert \leq
KF_{n_{s}-m,m}\left( f,x\right) _{X^{p}}\sum_{\nu =0}^{n_{s+1}-n_{s}-1}\frac{%
1}{m+\nu +1}
\end{equation*}%
and since \ $\frac{1}{2}F_{n_{s}-m,m}\left( f,x\right) _{X^{p}}\leq
F_{n_{s+1}-m-1,m}\left( f,x\right) _{X^{p}}$ we have%
\begin{eqnarray*}
\left\vert \frac{1}{m+1}\tau _{n_{s},m}f\left( x\right) \right\vert &\leq
&2KF_{n_{s+1}-m-1,m}\left( f,x\right) _{X^{p}}\sum_{\nu =0}^{n_{s+1}-n_{s}-1}%
\frac{1}{m+\nu +1} \\
&\leq &2K\sum_{\nu =0}^{n_{s+1}-n_{s}-1}\frac{F_{n_{s}-m+\nu ,m}\left(
f,x\right) _{X^{p}}}{m+\nu +1}
\end{eqnarray*}

Consequently,%
\begin{eqnarray*}
&&\sum_{s=0}^{t-1}\left\vert \frac{1}{m+1}\tau _{n_{s},m}f\left( x\right)
\right\vert \\
&\leq &2K\sum_{s=0}^{t-1}\sum_{\nu =0}^{n_{s+1}-n_{s}-1}\frac{F_{n_{s}-m+\nu
,m}\left( f,x\right) _{X^{p}}+F_{n_{s}-m+\nu ,\nu }\left( f,x\right) _{X^{p}}%
}{m+\nu +1}
\end{eqnarray*}%
Since \ $n_{s+1}-n_{s}\leq 2n+m-n-1=n+m-1$ \ for all \ $s\leq t-1$ ,
changing the order of summation we get%
\begin{eqnarray*}
&&\sum_{s=0}^{t-1}\left\vert \frac{1}{m+1}\tau _{n_{s},m}f\left( x\right)
\right\vert \\
&\leq &2K\sum_{\nu =0}^{n+m-1}\frac{1}{m+\nu +1}\sum_{s:n_{s+1}-n_{s}>\nu }%
\left[ F_{n_{s}-m+\nu ,m}\left( f,x\right) _{X^{p}}+F_{n_{s}-m+\nu ,\nu
}\left( f,x\right) _{X^{p}}\right] .
\end{eqnarray*}

Using the inequality%
\begin{equation*}
F_{n_{s+1}-m,\nu }\left( f,x\right) _{X^{p}}\leq \frac{1}{2}F_{n_{s}-m,\nu
}\left( f,x\right) _{X^{p}}\text{ \ \ for \ }\left\{ 
\begin{array}{c}
\nu =0,1,2,...,n_{s+1}-n_{s}-1 \\ 
s=0,1,2,...,t-2%
\end{array}%
\right.
\end{equation*}%
we obtain%
\begin{eqnarray*}
&&\sum_{s:n_{s+1}-n_{s}>\nu }\left[ F_{n_{s}-m+\nu ,m}\left( f,x\right)
_{X^{p}}+F_{n_{s}-m+\nu ,\nu }\left( f,x\right) _{X^{p}}\right] \\
&=&F_{n_{p}-m+\nu ,m}\left( f,x\right) _{X^{p}}+F_{n_{p}-m+\nu ,\nu }\left(
f,x\right) _{X^{p}} \\
&&+\sum_{s\geq p+1:n_{s+1}-n_{s}>\nu }\left[ F_{n_{s+1}-m+\nu ,m}\left(
f,x\right) _{X^{p}}+F_{n_{s+1}-m+\nu ,\nu }\left( f,x\right) _{X^{p}}\right]
\\
&\leq &F_{n_{p}-m+\nu ,m}\left( f,x\right) _{X^{p}}+F_{n_{p}-m+\nu ,\nu
}\left( f,x\right) _{X^{p}} \\
&&+\sum_{s:s\geq p+1}F_{n_{s}-m,m}\left( f,x\right) _{X^{p}}\left[
F_{n_{s}-m,m}\left( f,x\right) _{X^{p}}+F_{n_{s}-m,\nu }\left( f,x\right)
_{X^{p}}\right] \\
&\leq &F_{n_{p}-m+\nu ,m}\left( f,x\right) _{X^{p}}+F_{n_{p}-m+\nu ,\nu
}\left( f,x\right) _{X^{p}}+2\left[ F_{n_{p+1}-m,m}\left( f,x\right)
_{X^{p}}+F_{n_{p+1}-m,\nu }\left( f,x\right) _{X^{p}}\right] \\
&\leq &3\left[ F_{n_{p}-m+\nu ,m}\left( f,x\right) _{X^{p}}+F_{n_{p}-m+\nu
,\nu }\left( f,x\right) _{X^{p}}\right] \text{ ,}
\end{eqnarray*}%
where \ $p$ \ denote the smallest index \ $s$ \ having the property \ $%
n_{s+1}-n_{s}>\nu .$ Hence%
\begin{eqnarray*}
\sum_{s=0}^{t-1}\left\vert \frac{1}{m+1}\tau _{n_{s},m}f\left( x\right)
\right\vert &\leq &K\sum_{\nu =0}^{n+m-1}\frac{F_{n-m+\nu ,m}\left(
f,x\right) _{X^{p}}+F_{n-m+\nu ,\nu }\left( f,x\right) _{X^{p}}}{m+\nu +1} \\
&\leq &K\sum_{\nu =0}^{n}\frac{F_{n-m+\nu ,m}\left( f,x\right)
_{X^{p}}+F_{n-m+\nu ,\nu }\left( f,x\right) _{X^{p}}}{m+\nu +1}.
\end{eqnarray*}

and our proof follows.

$\blacksquare $

\subsection{Proof of Theorem 3}

The proof follows by the obvious inequality%
\begin{equation*}
\left\Vert E_{n}\left( f,x;\delta \right) _{C}\right\Vert _{C}\leq
E_{n}\left( f\right) _{C}\text{ .}
\end{equation*}%
$\blacksquare $


\begin{thebibliography}{9}
\bibitem{ABT} S. Aljan\v{c}i\v{c}, R. Bojanic and M. Tomi\'{c}, On the
degree of convergence of Fej\'{e}r-Lebesgue sums, L'Enseignement
Mathematique, Geneve, Tome XV (1969) 21-28.

\bibitem{BN} P. L. Butzer, R.J. Nessel, Fourier analysis and approximation,
Basel und Stuttgart 1971.

\bibitem{L} L. Leindler, Sharpening of Ste\v{c}kin's theorem to strong
approximation, Analysis Math. 16 (1990), 27-38.

\bibitem{WL} W. Lenski, Pointwise best approximation and de la Vall\'{e}%
e-Poussin means, submited.

\bibitem{S} S. B. Ste\v{c}kin, On the approximation of periodic functions by
de la Vall\'{e}e Poussin sums, Analysis Math. 4 (1978), 61-74.

\bibitem{TM} N. Tanovi\'{c}-Miller, On some generalizations of the Fej\'{e}%
r-Lebesgue theorem, Boll. Un. Mat. Ital. B(6) 1 (1982), no. 3, 1217-1233.
\end{thebibliography}
\end{document}